\documentclass{acm_proc_article-sp}
\usepackage{amsmath,amsfonts}%,amsthm}
\usepackage{amssymb}

\usepackage{epsfig}

\usepackage{graphics}
\usepackage{graphicx}

\usepackage{latexsym}
\usepackage{graphics}
\usepackage{amsmath}
\usepackage{xspace}
\usepackage{amssymb}
\usepackage{color}
\usepackage{cite}
\usepackage{latexsym}
\usepackage{graphics}
\usepackage{amsmath}
\usepackage{xspace}
\usepackage{amssymb}
\usepackage{epsfig}
\usepackage{cite}
\usepackage{url}
\usepackage{bm}
\usepackage{graphicx}
\usepackage{algorithm}
\usepackage[noend]{algorithmic}
\usepackage{subfig}
\usepackage{amssymb, amsmath,graphicx,charter, latexsym}
\usepackage{enumerate}
\usepackage{comment}
\usepackage{caption}
\usepackage{subfig}
\usepackage{tikz}
\usetikzlibrary{arrows}
\usepackage{adjustbox}

\DeclareMathOperator*{\argmax}{arg\,max}

\usepackage{soul}

\newtheorem{theorem}{Theorem}

\newtheorem{lemma}[theorem]{Lemma}

\newtheorem{proposition}[theorem]{Proposition}
\newtheorem{definition}[theorem]{Definition}
\allowdisplaybreaks[4]
\usepackage{bbm}

\newcommand{\beql}[1]{\begin{equation}\label{#1}}
\newcommand{\eeql}{\end{equation}}
\newcommand{\eqn}[1]{(\ref{#1})}

\newcommand{\R}{\mathbb{R}}     %%%%%%%%
\newcommand{\cb}{{\cal B}}

\newcommand{\bx}{\boldsymbol{x}}
\newcommand{\by}{\boldsymbol{y}}
\newcommand{\bz}{\boldsymbol{z}}

\newcommand{\blambda}{\boldsymbol{\lambda}}
\newcommand{\bmu}{\boldsymbol{\mu}}
\newcommand{\bgamma}{\boldsymbol{\gamma}}
\newcommand{\bnu}{\boldsymbol{\nu}}

\newcommand{\bzero}{\boldsymbol{0}}

\newcommand{\bA}{\boldsymbol{A}}
\newcommand{\bS}{\boldsymbol{S}}

\newcommand{\bV}{\boldsymbol{V}}
\newcommand{\bW}{\boldsymbol{W}}
\newcommand{\bQ}{\boldsymbol{Q}}
\newcommand{\bY}{\boldsymbol{Y}}
\newcommand{\bG}{\boldsymbol{G}}

\newcommand{\sL}{{\cal L}}

\newcommand{\eeq}{\end{equation}}
\newcommand{\beqal}[1]{\begin{eqnarray}\label{#1}}
\newcommand{\eeqa}{\end{eqnarray}}
\newcommand{\eeqno}{\end{displaymath}}

\newcommand{\bb}{\boldsymbol{b}}
\newcommand{\ba}{\boldsymbol{a}}
\begin{document}
\title{MaxWeight Scheduling: Asymptotic Behavior of Unscaled Queue-Differentials in Heavy Traffic
}            
%\author{R. Singh,~A. Stolyar}

\numberofauthors{2}

\author{
\alignauthor Rahul Singh\\
       \affaddr{ECE Department}\\
       \affaddr{Texas A\&M University}\\
       \affaddr{College Station, TX 77843}\\
       \email{rsingh1@tamu.edu}
\alignauthor Alexander L. Stolyar\\
      \affaddr{ISE Department}\\
       \affaddr{Lehigh University}\\
       \affaddr{Bethlehem, PA 18015}
       \email{stolyar@lehigh.edu}
                     }

\maketitle

\begin{abstract}
The model is a ``generalized switch", serving multiple traffic flows in discrete time. The switch uses MaxWeight algorithm to make a service decision (scheduling choice) at each time step, which determines the probability distribution of the amount of service that will be provided. We are primarily motivated by the following question: in the heavy traffic regime, when the switch load approaches critical level, will the service processes provided to each flow remain ``smooth"  (i.e., without large gaps in service)? Addressing this question reduces to the analysis of the asymptotic behavior of the unscaled queue-differential process in heavy traffic. We prove that the stationary regime of this process converges to that of a positive recurrent Markov chain, whose structure we explicitly describe. This in turn implies asymptotic ``smoothness" of the service processes.
\end{abstract}

 % NOT required for Proceedings

\section{Introduction}

Suppose we have a system in which several data traffic flows share a common transmission medium (or channel).
Sharing means that in each time slot a scheduler chooses a transmission mode -- the subset the flows to serve and corresponding 
transmission rates; the outcome of each transmission (the number of successfully delivered packets) is random.
Scheduler has two key objectives: (a) the time-average (successful) transmission rate of each flow $i$ has to be at least some $\lambda_i>0$;
(b) the successful transmissions for each flow need to be spread out "smoothly" in time -- without large time-gaps between succesful transmissions. Such models arise, for example, when the goal is {\em timely delivery} of information over a shared wireless
channel \cite{hou}. 

A very natural way to approach this problem is to treat the model as a queueing system, where services (transmissions) are 
controlled by a so called MaxWeight scheduler (see \cite{tass,c1, c2}), which serves a set of {\em virtual queues} (one for each traffic flow), each receiving
new work at the rate $\lambda_i$. (See e.g. \cite{akrsvw_commag}.) This automatically achieves objective (a), if this is feasible at all; MaxWeight 
is known to be {\em throughput optimal} -- stabilize the queues if this is feasible at all. The MaxWeight stability results, however,
do not tell whether or not the objective (b) is achieved. Specifically, when the system is heavily loaded, i.e. the vector 
$\blambda=(\lambda_i)$ is within the system {\em rate region} $\bV$, but close to its boundary, the steady-state queue lengths under 
MaxWeight are necessarily large, and it is conceivable that this may result in large time-gaps in service for individual flows.
(Note that, if (a) and (b) are the objectives and the queues are virtual, the large queue lengths in themselves are not an issue.
As long as (a) and (b) are achieved, minimizing the queue lengths is not important.)
Our main results show that this is {\em not} the case. Namely, in the heavy traffic regime, when $\blambda \to \blambda^\star$,
where  $\blambda^\star$ is a point on the outer boundary of rate region $\bV$, the service process remains "smooth", in the sense
that its stationary regime converges to that of a positive recurrent Markov chain, whose structure is given explicitly.

To obtain "clean" convergence results, we assume that the amount of new work arriving in the queues in each time slot is random and has {\em continuous} distribution. (The amounts of service are random, but discrete.)
Under this assumption, the state spaces of the processes that we consider are continuous. On one hand, this makes 
the analysis more involved (because the notion of positive recurrence is more involved for a continuous state space,
as opposed to a countable one). But on the other hand, this makes
all stationary distributions absolutely continuous w.r.t. the corresponding Lebesgue measure, making it easier to prove convergence.
We emphasize that the assumption of continuous distribution of the arriving work is non-restrictive; if we create 
virtual queues, artificially, for the purpose of applying MaxWeight algorithm, the structure of the virtual arrival
process is within our control.
 
The problem essentially reduces to analysis of stationary versions of the {\em queue-differential} process $\bY$, which is
the projection of the (weighted) queue length process on the subspace $\bnu_\perp$, orthogonal to the outer {\em normal cone} $\bnu$ to
the rate region $\bV$ at the point $\blambda^\star$. As we show, in the heavy-traffic limit, in steady-state,
the values of the queue-differential process $\bY$
uniquely determine the decisions chosen by MaxWeight scheduler. 
Note that the process $\bY$ is obtained by projection only, without any scaling depending on the system load.

The model that we consider is essentially a "generalized switch" of \cite{c1}. Some features of our model,
namely random service outcome and continuous amounts of arriving work, as well as the objective (b), 
are motivated by applications such as timely delivery 
of packets of multiple flows over a shared wireless channel \cite{hou}. The model of \cite{hou} is a special case of ours;
paper \cite{hou} introduces a {\em debt scheme} and proves that it achieves the throughput objective (a);
the objective (b) is not considered in \cite{hou}. 

The analysis of MaxWeight stability has a long history, starting from the seminal paper \cite{tass}, which introduced\\ MaxWeight;
heavy traffic analysis of the algorithm originated in \cite{c1}. (See, e.g., \cite{c2} for an extensive recent review of
\\ MaxWeight literature.) 

The line of work most closely related to this paper, is that in \cite{c2, atilla1, atilla2}. 
Paper \cite{c2} studies MaxWeight under heavy traffic regime and under the additional assumption that 
the normal cone $\bnu$ is one-dimensional, i.e. it is a ray. (The latter assumption is usually referred to in the literature
as {\em complete resource pooling} (CRP).)
Paper \cite{c2} shows, in particular, the stationary distribution tightness of what we call the queue-differential process $\bY$ in heavy traffic. Part of our analysis is also showing the stationary distribution tightness of $\bY$ -- it is analogous to that in \cite{c2} (and we also 
borrow a lot of notation from \cite{c2}). Besides the difference in models, our proof of tightness
is more general in that it applies to non-CRP case -- this more general argument is close to that used in \cite{aks2007stoc}.
From the tightness of stationary distributions, using the structure of the corresponding continuous state space,
we obtain the convergence of the stationary version of (non-Markov) process $\bY$ to that of a positive recurrent Markov chain,
whose structure we explicitly describe.

Papers \cite{atilla1, atilla2} consider objective (b) in the heavy traffic\\ regime. They introduce a modification of MaxWeight, 
called {\em regular service guarantee} (RSG) scheme, 
which 
explicitly tracks the service time-gaps for each flow to dynamically increase the scheduling priority of flows with large current time-gaps.
The papers prove that RSG, under certain parameter settings, preserves heavy-traffic queue-length minimization properties 
of MaxWeight, under CRP condition; at the same time, the papers demonstrate via simulations 
that RSG improves smoothness (regularity) of the service process. Recall that in this paper we focus on the "pure" MaxWeight,
without CRP, and formally show the service process smoothness in the heavy traffic limit.

\iffalse
As mentioned above, our results show that, in the heavy-traffic limit, the stationary version of  the
process is such that the following property holds: (c) {\em scheduler decisions are uniquely determined 
by the values of the queue-differential process $\bY$.}
We want to point out that property (c), while very intuitive in the case of the CRP condition, is less intuitive when the CRP
condition does not hold. Indeed, it is well known that in the heavy traffic limit, the queue length vector stays "close" 
to the normal cone $\bnu$, whether CRP holds or not; in the case of CRP this 
(along with the fact that queue length vector is large) is sufficient to conclude (c).
When CRP does not hold, i.e. the normal cone $\bnu$ is $d$-dimensional with $d\ge 2$, property (c) 
requires, in addition, that the queue length vector projection on the normal cone lies "far" from 
the relative boundary of the cone.
\fi

The rest of the paper is organized as follows. The formal model is presented in Section~\ref{systemmodel}.
Section~\ref{relevantbackground} describes the MaxWeight algorithm and the heavy traffic asymptotic regime.
Our main results, Theorems~\ref{th-smooth} and \ref{theorem2} are described in Section~\ref{problemmotivation}.
(Formal statement of Theorem~\ref{theorem2} is in Section~\ref{sec-Ystar}.)
The CRP condition is defined in Section~\ref{sec-crp}.
In Section~\ref{MTresults} we provide some necessary background and results for general state-space Markov chains.
Sections~\ref{sec-queue-length} -- \ref{sec-Ystar} we prove our results for the special case when CRP holds.
Finally, in Section~\ref{sec-no-crp} we show how the proofs generalize to the case when CRP does {\em not} necessarily hold.

\subsection{Basic notation}

\iffalse
'COM-RS' COMMENTS ARE FOR YOU. THE REST OF COMs ARE FOR MYSELF.  --SASHA ENDCOM

COM-RS: Please put paper in the required Sigmetrics style. PLEASE (!) do it carefully, so that text, formulas, plots, etc, fit margins and look nice. 
 YOU DO NOT HAVE TO DO ANYTHING WITH THE 
EXISTING SECTIONS, EXCEPT PUTTING EVERYTHING IN SIGMETRICS STYLE/FORMAT.
ENDCOM

COM-RS: Add simulation plots, and a description of what's on them, as a separate section at the end. 

COM-RS: Put additional references that we discussed  into the reference list:
 credit schemes; I remember there were more references to Srikant et al; some relevant references to random access, because 
we use it for comparison. Write a short literature review -- as a separate section at the beginning. ENDCOM

COM-RS When you write:  'at least', 'at most', 'bold font', are written as two words. (Check correct spelling in google, if necessary.)
'MaxWeight' is one word.  It is 'first', not 'firstly'. Vector components are 'components', not 'entrires'. 
ENDCOM

COM-RS Again, PLEASE (!!!!!) do a careful job when you write. Check that the text is clean and makes sense, before sending to me.
\fi

Elements of a Euclidean space $\mathbb{R}^{N}$ will be viewed as row-vectors, and written in bold font; $\|\boldsymbol{a}\|$ is the usual Euclidean norm of vector $\boldsymbol{a}$.
For two vectors $\boldsymbol{a}$ and $\boldsymbol{b}$, $\boldsymbol{a}\cdot \boldsymbol{b}$ denotes their scalar (dot) product; vector inequalities are understood componentwise;
zero vector and the vector of all ones are denoted 
$\boldsymbol{0}$ and $\boldsymbol{1}$,  respectively; $\boldsymbol{ab}$ will denote the vector obtained by componentwise multiplication;
if all components of  $\boldsymbol{b}$ are non-zero, $\boldsymbol{\frac{a}{b}}$ will denote the vector obtained by componentwise division;
statement ``$\boldsymbol{a}$ is a positive vector" means $\boldsymbol{a}>\boldsymbol{0}$. 
%If $\boldsymbol{a}\neq \boldsymbol{0}$, then $\hat{\boldsymbol{a}}$ will denote the unit vector in the direction of $\boldsymbol{a}$, i.e. $\hat{\boldsymbol{a}} =  \frac{\boldsymbol{a}}{||\boldsymbol{a}||}$; in particular, $\hat{\boldsymbol{1}}
%=\frac{1}{\sqrt{N}}\boldsymbol{1}$. 
The closed ball of radius $r$ centered at $\boldsymbol{x}$ is $B_{r}(\boldsymbol{x})$. The positive orthant of $\mathbb{R}^{N}$ is denoted
$\mathbb{R}^{N}_{+} = \{\boldsymbol{x}\in\mathbb{R}^{N}:\boldsymbol{x}\geq \boldsymbol{0}\}$. 

For numbers $a$ and $b$, we denote $a \vee b = \max(a,b)$, $a \wedge b = \min(a,b)$, $a^+=a \vee 0$. For vectors $\ba \le \bb$, we denote by $[\ba,\bb]$ the
rectangle $\times_{i=1}^N [a_i, b_i]$ in $\mathbb{R}^{N}$.

We always consider Borel $\sigma$-algebra $\cb(\mathbb{R}^{N})$ (resp. $\cb(\mathbb{R}^{N}_+)$ )
 on $\mathbb{R}^{N}$ (resp. $\cb(\R^N_+)$), when the latter is viewed as measurable space. 
 Lebesgue measure on $\mathbb{R}^{N}$ is denoted by 
$\sL$. 
When we consider a linear subspace of $\mathbb{R}^{N}$, we endow it with
the Euclidean metric
and the corresponding Borel $\sigma$-algebra and Lebesgue measure.

For a random process $\bW(t), ~t=0,1,2, \ldots$, we often use notation $\bW(\cdot)$ or simply $\bW$.

\section{System Model}
\label{systemmodel}

We consider a system of $N$ flows served by a ``switch", which evolves in discrete time $t=0,1,\ldots$. At the beginning of each time-slot, the scheduler has to choose from a finite number $K$ of ``service-decisions". If the service decision $k\in\left\{1,\ldots,K\right\}$ is chosen, then  independently of the past history 
the flows get an amount of service, given by a random non-negative vector.  
Furthermore, we assume that (if decision $k$ is chosen), there is a finite number $\mathcal{O}_{k}$ of possible service-vector outcomes, i.e. with probability $p^{k,j}, j = 1,\ldots,\mathcal{O}_{k}$, it is given by a non-negative vector $\boldsymbol{v}^{k,j} = (v^{k,j}_1, \ldots, v^{k,j}_N)$. 
The expected service vector for decision $k$ is denoted
$\boldsymbol{\mu}^{k} = (\mu^{k}_1, \ldots, \mu^{k}_N) = \sum_{j=1}^{\mathcal{O}_{k}} \boldsymbol{v}^{k,j}p^{k,j}$.
We assume that vectors $\boldsymbol{\mu}^{k}$ are non-zero and different from each other; and that for each $i$ there exists $k$ such that $\mu^{k}_i >0$.
We will use notations
$$
S_i^{max} = \max v^{k,j}_i ~~\mbox{over all $k$ and $j$}; ~~~\boldsymbol{S}^{max} = (S_1^{max}, \ldots, S_N^{max}).
$$
We denote by $\boldsymbol{S}(t)=(S_1(t),\ldots,S_N(t))$ the (random) realization of the service vector at time $t$, and call $\boldsymbol{S}(\cdot)$ the service process.

After the service at time $t$ is completed, a random  amount of work arrives into the queues, and it is given by a 
non-negative vector $\boldsymbol{A}(t)=(A_1(t),\ldots,A_N(t))$. The values of $\boldsymbol{A}(t)$ are i.i.d. across times $t$, and 
$\boldsymbol{A}(\cdot)$ is called the arrival process. The mean arrival rates of this process are given by vector 
%\\ COM USE $\E, \PP$ FOR EXPECTATION AND PROB THROUGOUT?
$$
\boldsymbol{\lambda} = (\lambda_1, \ldots, \lambda_N) = E \boldsymbol{A}(t).
$$
We will now make assumptions on the distribution of $\boldsymbol{A}(t)$.
The distribution is absolutely continuous w.r.t. Lebesgue measure, it is concentrated on the rectangle
$
[\bzero,\boldsymbol{A}^{\max}]
%\times_{i=1}^{N}[0,A^{\max}_{i}],
$
for some constant vector $\boldsymbol{A}^{\max} > \boldsymbol{S}^{\max}$; moreover,
on this rectangle the distribution density $f(\bx)$ is both upper 
and lower bounded by positive constants, i.e. $0 < \delta_* \le f(\bx) \le \delta^*$.

If $\boldsymbol{Q}(t)\triangleq \left(Q_{1}(t),\ldots,Q_{N}(t)\right)$ is the vector of queue lengths at time $t$, then for each 
$ i=1,\ldots,N$
 \begin{align}\label{eq:1}
Q_{i}(t+1) &= \left(Q_{i}(t) -S_{i}(t)\right)^{+}+ A_{i}(t), \notag\\
& = Q_{i}(t) + A_{i}(t) -S_{i}(t) + U_{i}(t),
\end{align}
where $U_{i}(t)=(S_i(t)-Q_i(t))^+$ is the amount of service ``wasted" by flow $i$ at time $t$.

\section{MaxWeight scheduling scheme. Heavy traffic regime}\label{relevantbackground}

\subsection{MaxWeight definition}

Let a vector $\boldsymbol{\gamma} = \left(\gamma_{1},\ldots,\gamma_{N}\right)>\boldsymbol{0}$ be fixed. 
%COM ONE WORD MaxWeight EVERYWHERE ENDCOM
MaxWeight scheduling algorithm chooses, at each time $t$,  a service decision 
\begin{equation}
\label{mwdef}
k \in \argmax_{l} \left(\left(\boldsymbol{\gamma Q}(t)\right)\cdot \boldsymbol{\mu}^{l}  \right);
\end{equation}
with ties broken according to any well defined rule. 

Under MaxWeight, the queue length process $\bQ(\cdot)$
is a discrete time Markov chain with (continuous) state space $\mathbb{R}_{+}^{N}$. 
System stability 
is understood as positive Harris recurrence
of this Markov chain. 

Denote the system {\em rate region} by
\begin{align}
\boldsymbol{V} \triangleq \left\{\boldsymbol{x}\in\mathbb{R}^{N}_{+}:\boldsymbol{x}\leq \sum\limits_{k}\psi_{k}\boldsymbol{\mu}^{k} \mbox{ for some } \psi_{k}\geq 0,\sum\limits_{k}\psi_{k}=1	\right\}
\end{align}
It is well known (see \cite{tass,c1,c2}) that, in general, under\\ MaxWeight the system is stable as long as the vector of mean arrival rates $\boldsymbol{\lambda}$ is such that $\boldsymbol{\lambda}<\boldsymbol{x}\in\boldsymbol{V}$.
(Scheduling rules having this property are sometimes called ``throughput-optimal".)
This is true for our model as well as will be shown in Section~\ref{sec-queue-length}. 
(Establishing this fact is not difficult, but it does not directly follow from previous work,
because we have continuous state space.)

%COM  NEED AT ALL? We will use notations
%$\gamma_{\max} =\max \gamma_i$ and $\gamma_{\min}=\min \gamma_i$.

\subsection{Heavy traffic regime}\label{ht}

We will consider a sequence of systems, indexed by $n\to \infty$, operating under MaxWeight scheduling. 
(Variables pertaining to $n$-th system will be supplied superscript $(n)$.)
The switch parameters will remain 
unchanged, but the distribution of $\bA^{(n)}(t)$ changes with $n$: 
namely, for each $n$ it has density $f^{(n)}$ which satisfies all conditions specified in Section~\ref{systemmodel},
and $f^{(n)}$ uniformly converges
to some density $f^*$. Note that, automatically, the limiting density  $f^{\star}$ (as well as each $f^{(n)}$)
satisfies bounds $0 < \delta_* \le f^{\star}(\bx) \le \delta^*$ in the rectangle $[\bzero,\bA^{max}]$,
and is zero elsewhere. The arrival process  $\bA^{\star}(\cdot)$, such that the distribution of $\bA^{\star}(t)$ has density $f^{\star}$, 
has the arrival rate vector $\blambda^{\star}$.
Correspondingly, $\blambda^{(n)} \to \blambda^{\star}$. 

We assume that $\blambda^{\star}>\bzero$ is a maximal element of rate region $\bV$, i.e. 
$\bx \ge \blambda^{\star}$ and $\bx \in \bV$ only when $\bx = \blambda^{\star}$.
Thus,  $\blambda^{\star}$ lies on the outer boundary of  $\bV$.
We further assume that for each $n$,
$\boldsymbol{\lambda}^{(n)}$ lies in the {\em interior} of $\bV$;
therefore, the system is stable for each $n$ (under the MaxWeight algorithm).

The (limiting) system, with arrival process $\bA^{\star}(\cdot)$ is called
critically loaded. 

%$\left(\boldsymbol{\lambda}^{(n)}-\boldsymbol{\lambda^{\star}}\right)\cdot \boldsymbol{\nu}<0$, i.e. the sequence of 
%$\boldsymbol{\lambda}^{(n)}$ approaches $\blambda^{\star}$ from the interior of $\bV$. 

\section{Main Results}
\label{problemmotivation}

Consider the sequence of systems described in Section~\ref{relevantbackground}, in the heavy traffic regime.
Under any throughput-optimal scheduling algorithm, for each $n$, the steady-state average amount of service provided to each flow $i$
is greater or equal to its arrival rate $\lambda_i$. (It may, and typically will, be greater if the wasted service is taken into account.)

We now define the notion of {\em asymptotic smoothness} of the steady-state service process.
Informally, it means the property that as the system load approaches critical, the steady state service processes are 
such that for each flow the probability of a $T$-long gap (without any service at all) uniformly vanishes, as $T\to\infty$.

For each $n$, consider the cumulative service process $\boldsymbol{G}^{(n)}(\cdot)$ in steady state. Namely,
\begin{align*}
\boldsymbol{G}^{(n)}(t) \triangleq \sum_{\tau=1}^{t} \boldsymbol{S}^{(n)}(\tau), ~~~t=1,2,\ldots
\end{align*}

\begin{definition}
We call the service process asymptotically \\smooth, if
\beql{eq888}
\max_i \lim_{T\to\infty}\left(\limsup_{n\to\infty} P\left(G^{(n)}_{i}(T)=0\right)\right)=0.
\end{equation}
%COM NEED? Otherwise, we call the service process asymptotically bursty.
\end{definition}

Our key result (Theorem~\ref{theorem2} in Section~\ref{sec-Ystar}) 
shows that a "queue-differential" process, which determines scheduling decisions in 
the system under MaxWeight in heavy traffic, is such that its stationary version converges to that of stationary positive Harris 
recurrent Markov chain, whose structure we describe explicitly. This result, in particular, will imply the following

\begin{theorem}
\label{th-smooth}
Consider the sequence of systems described in Section~\ref{relevantbackground}, in the heavy traffic regime.
Under MaxWeight scheduling, 
the service process is asymptotically smooth.
\end{theorem} 

The proof is given in Section~\ref{sec-Ystar}.

\section{Complete Resource Pooling\\ condition}
\label{sec-crp}

To improve exposition, we first give detailed proofs of our main results for the 
special case, when the following {\em complete resource pooling} (CRP) condition
holds. (In Section~\ref{sec-no-crp} we will show how the proof generalizes to the case without the
CRP condition.) Assume that vector $\blambda^{\star}$ is such that
there is a unique (up to scaling) 
outer normal vector $\bnu>\bzero$ to $\bV$ at point $\blambda^{\star}$; we choose $\bnu$ so that $\|\bnu\|=1$.
Denote by 
\begin{align}\label{vstar}
\boldsymbol{V}^{\star} \triangleq \argmax_{\boldsymbol{x}	\in\boldsymbol{V}} 	\boldsymbol{\nu}\cdot \boldsymbol{x}
\end{align}
the outer face of $\bV$ where $\blambda^{\star}$ lies. Given our assumptions on $\blambda^{\star}$, it  lies in the relative 
interior of $\bV^\star$.

By $\bnu_\perp$ we denote the subspace of $\mathbb{R}^{N}$
orthogonal to $\bnu$. 
 For any vector $\boldsymbol{a}$,  we denote by $\boldsymbol{a}_{\star} \triangleq \left(\boldsymbol{a}\cdot {\boldsymbol{\nu}}\right){\boldsymbol{\nu}}$ its orthogonal projection 
on the (one-dimensional) subspace spanned by $\bnu$, and by 
$\boldsymbol{a}_{\perp} \triangleq \boldsymbol{a} - \boldsymbol{a}_{\star}$ 
its orthogonal projection on the $(N-1)$-dimensional subspace $\bnu_\perp$.

The following observations and notations will be useful. There is a $\delta>0$ such that the entire set
\begin{align}\label{deltaball}
B^{\delta}_{\boldsymbol{\lambda^{\star}}} \triangleq \{\boldsymbol{y} \in \boldsymbol{V}^{\star}: \|\boldsymbol{y} -\boldsymbol{\lambda}^{\star} \| \leq \delta   \},
\end{align}
also lies in the relative interior of $\boldsymbol{V}^{\star}$.
% Denote $\boldsymbol{\nu}\cdot\boldsymbol{\lambda}^{\star}=c$. Clearly for $\boldsymbol{x}\in \boldsymbol{V}^{\star}$, $\boldsymbol{\nu}\cdot \boldsymbol{x}=c$.

\section{Background on general-state-\\space discrete-time \\Markov Chains}
\label{MTresults}

We will briefly discuss some notions and results from~\cite{meyn1992stability} and \cite{hajek} 
on the stability of discrete time Markov Chains (MC), which will be used in later sections. Throughout this section we will assume that the Markov Chain $\boldsymbol{\Phi} = \left\{\Phi(0),\Phi(1),\ldots \right\}$ is evolving on a locally compact separable metric space $\boldsymbol{X}$ whose Borel $\sigma$-algebra will be denoted by $\mathcal{B}$. $P_{\eta}$ and $E_{\eta}$ are used to denote the probabilities and expectations conditional on $\Phi_{0}$ having distribution $\eta$, while $P_{\boldsymbol{x}}$ and $E_{\boldsymbol{x}}$ are used when $\eta$ is concentrated at $\boldsymbol{x}$. The transition function of $\boldsymbol{\Phi}$ is denoted by $P(\boldsymbol{x},A), \boldsymbol{x}\in \boldsymbol{X}, A\in \mathcal{B}$. The iterates $P^{t},~ t=0,1,2, \ldots$, 
%\\COM DEFINE $\Z_+$  ? ENDCOM \\
are defined inductively by
\begin{align*}
P^{0}\triangleq I, P^{t}\triangleq PP^{t-1}, t\geq 1,
\end{align*}
where $I$ is the identity transition function.

\begin{definition}
(i) $\phi$-irreducibility: A Markov Chain $\boldsymbol{\Phi} = \left\{\Phi(0),\Phi(1),\ldots \right\}$ is called $\phi$ irreducible if there exists a finite measure $\phi$ such that
$\sum\limits_{k=1}^{\infty} P^{k}(\bx,A)>0$ for all $\bx\in \boldsymbol{X}$ whenever $\phi(A)>0$. Measure $\phi$ is called an irreducibility measure.

(ii) Harris Recurrence: If $\boldsymbol{\Phi}$ is $\phi$-irreducible and $P_{x}(\Phi(t)\in A \quad \mbox{i.o.})\equiv 1$ 
whenever $\phi(A)>0$, then $\boldsymbol{\Phi}$ is called Harris recurrent. [Abbreviation 'i.o.' means 'infinitely often'.]

(iii) Invariant Measure: A $\sigma$-finite measure $\pi$ on $\mathcal{B}$ with the property 
\begin{align*}
\pi\left\{A\right\} = \pi P\left\{A\right\} \triangleq \int \pi(d\bx)P(\bx,A), \forall A\in \mathcal{B},
\end{align*}
is called an invariant measure.

(iv) Positive Harris Recurrence: If $\boldsymbol{\Phi}$ is Harris Recurrent with a finite invariant measure $\pi$, then it is called positive Harris Recurrent.

(v) Boundedness in Probability: If for any $\epsilon>0$ and any $\bx\in\boldsymbol{X}$, there exists a compact set $D$ such that 
\begin{align}\label{bip}
\lim\inf\limits_{t\to\infty} P_{\bx} (\Phi(t)\in D) \geq 1-\epsilon,
\end{align}
then the Markov process $\boldsymbol{\Phi}$ is called bounded in probability.

(vi) Small Sets: A set $C$ is called small if for all $\bx\in C$ and some integer $l\ge 1$, we have
\beql{smallsetdef}
P^{l}(\bx,\cdot) \geq \nu(\cdot),
\end{equation}
where $\nu(\cdot)$ is a sub-probability measure, i.e. $\nu(\boldsymbol{X}) \leq 1$.

(vii) For a probability distribution $\boldsymbol{a} = \left(a_{1},a_{2}\ldots\right)$ on $\{1,2,\ldots\}$, the Markov transition function  $K_{\boldsymbol{a}}$ is defined as
\begin{align*}
K_{\boldsymbol{a}} \triangleq \sum\limits_{i=1}^{\infty} a_{i}P^{i}.
\end{align*}

(viii) Petite Sets: A set $A\in\mathcal{B}$ and a sub-probability measure $\psi$ on $\mathcal{B}(\boldsymbol{X})$ are called petite if for some probability distribution $\boldsymbol{a}$ on $\{1,2,\ldots\}$ we have
\begin{align*}
K_{\boldsymbol{a}}(\bx,\cdot)\geq \psi(\cdot), \forall \bx\in A.
\end{align*}

(ix) Non-evanescence: A Markov chain $\boldsymbol{\Phi}$ 
is called non-evanescent if $P_{\bx} \{\boldsymbol{\Phi}\to\infty\} =0$ 
for each $\bx \in \boldsymbol{X}$. 
[Event $\{\boldsymbol{\Phi}\to\infty \}$ consists of the outcomes such that the sequence 
$\Phi(t)$ visits any compact set at most a finite number of times.]

\end{definition} 

The following proposition states some results from~\cite{meyn1992stability}.

\begin{proposition}
\label{MTtheorem}
(i) If a set $A$ is small and for some probability distribution $\ba$ on $\{1,2,\ldots\}$ and a set $B\in\mathcal{B}$, we have
\begin{align}\label{petite1}
\inf_{\bx\in B }  K_{\ba}(\bx,A)>0,
\end{align}
then $B$ is petite. 
%We note that a special case under which the inequality~(\ref{petite1}) holds true is when the probability to hit the set $A$ starting from any point in set $B$ is lower bounded by a positive quantity. 

(ii) Suppose that every compact subset of $\boldsymbol{X}$ is petite. Then $\boldsymbol{\Phi}$ is positive Harris recurrent if and only if it is bounded in probability. 
%Also $\boldsymbol{\Phi}$ is positive Harris recurrent iff it is non-evanescent.  

(iii) Suppose that every compact subset of $\boldsymbol{X}$ is petite. Then $\boldsymbol{\Phi}$ is Harris recurrent if and only if it is non-evanescent. 
\end{proposition}

%If a random process $\boldsymbol{Y} = Y(0),Y(1),\ldots$ satisfies the condition
%\begin{align}\label{expb}
%P(Y(n)\geq x| Y(0)= b) \leq \rho^{n} \exp(\eta (b-x)) +\frac{1-\rho^{n}}{1-\rho}D\exp(\eta (\kappa-x)), \forall b
%\end{align}
%for some $\eta>0, 1>\rho>0$ and $D,\kappa,a$, then we will say that the process is \textit{uniformly bounded by an exponential}. 

The following result is form from~\cite{hajek}. It is stated in a form convenient for its application in this paper.
\begin{proposition}\label{th:hajek}
Let $L(\boldsymbol{x})$ be a non-negative (Lyapunov) function such that the Markov process $\boldsymbol{\Phi}$ satisfies the following two conditions,
for some positive constants $\kappa, \delta, D$:\\
(a) $E\left[L(\Phi(t+1))-L(\Phi(t))| \Phi(t) =\boldsymbol{x}\right]<-{\delta}$, for any state $\bx$ such that  $L(\bx)\ge \kappa>0$.\\
(b) $| L(\Phi(t+1))-L(\Phi(t))| <D$.\\
Then there exist constants $\eta>0$ and $0<\rho<1$ such that
$$
P(L(\Phi(t))\geq u ~|~ L(\Phi(0)) = b) \le
$$
\beql{hajek1}
\rho^{t} \exp(\eta (b-u))
+\frac{1-\rho^{t}}{1-\rho}D\exp(\eta (\kappa-u)), ~u\ge 0.
\end{equation}
\end{proposition}

\section{Queue length process}
\label{sec-queue-length}

Recall that $\bQ^{(n)}(\cdot)$ is the queue length process for the $n$-th system under MaxWeight.
In this section we prove that for all $n$, the process $\boldsymbol{Q}^{(n)}(\cdot)$ is positive Harris recurrent.
The proof uses a Lyapunov drift argument which is fairly standard (in fact, there is more than one way to prove stability 
of $\boldsymbol{Q}^{(n)}(\cdot)$), except, since our state space is continuous, as a first step we will show that all compact sets are petite.

Some  simple preliminary observations  are given in the following lemma. 

\begin{lemma}
\label{lemma1}
(i) The points $\bx \in \R_+^N$, such that\\ $k\in \argmax_{\ell} (\bgamma \bx) \cdot \bmu^l$  is non-unique, form a set of zero Lebesgue measure.
Moreover, if $\bx > \bzero$ is such that  $k \in \argmax_{\ell} (\bgamma \bx) \cdot \bmu^l$  is unique, then for a sufficiently small $\epsilon>0$ the decision $k$
is also the unique element of $\argmax_{\ell} (\bgamma \by) \cdot \bmu^l$ for all $\by \in B_{\epsilon}(\bx)$.\\
(ii) The one-step transition function $P^{(n)}(\bx,\cdot)$ of the process $\bQ^{(n)}(\cdot)$ is such that,  uniformly in $n$ and  $\bx \in \R_+^N$, the distribution
$P^{(n)}(\bx,\cdot)$ is absolutely continuous with the density upper bounded by $\delta^*$ and, in the rectangle 
 $[\bzero, \bA^{max} - \bS^{max}]$, lower bounded
by $\delta_*$.
\end{lemma}
\begin{proof}
Statement (i) easily follows from the finiteness of the set of decisions $k$. Statement (ii) easily follows from the assumptions on the arrival process distribution
and the fact that $\bS^{max} < \bA^{max}$.
\end{proof}

\begin{lemma}\label{lemma2}
For any $\bx > \bzero$, there exists $\epsilon>0$ such that the set $B_{\epsilon}(\bx)$ is small for the process $\bQ^{(n)}(\cdot)$.
\end{lemma}
\begin{proof}
Consider rectangle\\ $H=[\bx+(1/3)(\bA^{max} - \bS^{max}), \bx+(2/3)(\bA^{max} - \bS^{max})]$.
Choose $\epsilon>0$ small enough, so that $\epsilon < (1/3)\min_i (A^{max}_i - S^{max}_i)$ and  $\epsilon < \min_i x_i$.
Then, $B_{\epsilon}(\bx)$ lies in the interior of $\R_+^N$ and every point in $B_{\epsilon}(\bx)$ is strictly smaller than any point in $H$.
Lemma~\ref{lemma1}(ii) implies that for any $\by \in B_{\epsilon}(\bx)$, the distribution $P^{(n)}(\by,\cdot)$ has a density lower bounded 
by $\delta_*$ in $H$.
\end{proof}

 \begin{lemma}
\label{prop2}
For the Markov process $\boldsymbol{Q}^{(n)}(\cdot)$, any compact set is petite. 
\end{lemma}
\begin{proof}
Consider a compact set $G\subset \R_+^N$; of course, $G$ is bounded. Fix arbitrary $\bx > 0$ and pick $\epsilon>0$ small enough, so that 
$B_{\epsilon}(\bx)$ is small and lies in the interior of $\R_+^N$. Pick small $\delta>0$ such that any point in $\{\|y\| \le \delta\}$ is strictly less than
any point in $B_{\epsilon}(\bx)$.

It is easy to verify that there exists an integer $\tau>0$ such that the following holds uniformly in $\bQ^{(n)}(0) \in G$:
\beql{eq-prob-down}
P\{\|\bQ^{(n)}(\tau)\| \le \delta\} \ge \alpha ~\mbox{for some}~ \alpha>0.
\end{equation}
Indeed, suppose first that for all 
$t=0,1, \ldots$, $\bA^{(n)}(t) = \bzero$. (This is a probability zero event, of course, but let's consider it anyway.)
Then, for any $\delta_3>0$ there exist $\delta_1, \delta_2 > 0$, such that the following holds:
with probability at least some $\delta_1>0$, the norm 
$\|\bQ^{(n)}(t)\|$ decreases at least by some $\delta_2>0$, at each time $t$ when $\|\bQ^{(n)}(t)\| \ge \delta_3 >0$.
This implies that for some $\tau>0$ and $\delta_4>0$, $\bQ^{(n)}(0) \in G$ implies $P\{\|\bQ^{(n)}(\tau)\| \le \delta_3\} \ge \delta_4$.
Now, using this and the fact that with a positive probability $\bA^{(n)}(t)$ can be ``very close to $\bzero$,'' we can easily 
establish property \eqn{eq-prob-down}. (We omit rather trivial details.)

Next, it is easy to show that there  exists an integer $\tau_1>0$ such that the following holds uniformly in $\|\bQ^{(n)}(0)\| \le \delta$:
\beql{eq-prob-up}
P\{\|\bQ^{(n)}(\tau_1)\| \in B_{\epsilon}(\bx)\} \ge \alpha_1 ~\mbox{for some}~ \alpha_1>0.
\end{equation}
Here we use Lemma~\ref{lemma1}(ii), which shows that at each time step the distribution of the increments of $\boldsymbol{Q}^{(n)}(\cdot)$ has a density
lower bounded by $\delta_*$ in $[\bzero, \bA^{max} - \bS^{max}]$.

From \eqn{eq-prob-down} and \eqn{eq-prob-up} we see that uniformly in $\bQ^{(n)}(0) \in G$, 
$P\{\|\bQ^{(n)}(\tau+\tau_1)\| \in B_{\epsilon}(\bx)\} \ge \alpha \alpha_1$. Application of Theorem~\ref{MTtheorem}(i)  shows that $G$ is petite (and, moreover, that it is small).
\end{proof}

To prove stability, we will apply Proposition~\ref{MTtheorem} which requires the following
\begin{lemma}\label{prop1}
Consider the scalar projection $\|\sqrt{\boldsymbol{\gamma}}\boldsymbol{Q}^{(n)}(\cdot)\|,$ $ t = 0,1,\ldots$ of the the Markov process $\boldsymbol{Q}^{(n)}$ starting with a fixed
 initial state $\boldsymbol{Q}^{(n)}(0)$, such that $\|\sqrt{\boldsymbol{\gamma}}\boldsymbol{Q}^{(n)}(0)\| = b$. Then, uniformly 
on all large $n$ we have,
\begin{align}\label{hajek2}
P(\|\sqrt{\boldsymbol{\gamma}}\boldsymbol{Q}^{(n)}(t)\|\geq u) &\leq \rho^{t} \exp(\eta (b-u))\notag\\ 
&\qquad+\frac{1-\rho^{t}}{1-\rho}D\exp(\eta (\kappa-u)), ~u \ge 0,
\end{align}
for some constants $\eta,\kappa,D>0$ and $1>\rho>0$ which depend on $n$.
Consequently, the process $\boldsymbol{Q}^{(n)}(\cdot)$ is bounded in probability.
\end{lemma}

 \begin{proof}
We will use notation $L(\boldsymbol{x})  = \|\sqrt{\boldsymbol{\gamma}} \boldsymbol{x}\|$. Then \\$L(\boldsymbol{Q}^{(n)}(0))=b$.
Clearly, $|L(\boldsymbol{Q}^{(n)}(t+1))-L(\boldsymbol{Q}^{(n)}(t))|$ is uniformly bounded by a constant, given our assumptions on the arrival and service processes.
%$N\sqrt{\gamma_{\max}}\max\limits_{i}\left(A_{i}^{\max} + S^{\max}_{i}\right)$. 
We will show that the drift (average increment) of $L(\bQ^{(n)}(t+1))) - L(\bQ^{(n)}(t)))$ is upper bounded by some $-\tilde{\delta} < 0$ when 
$\|L(\bQ^{(n)}(t))\| \ge \kappa$ for some $\kappa>0$. 
  
 Consider a fixed $\boldsymbol{Q}^{(n)}(t)$ and denote $\Delta L = E[L(\boldsymbol{Q}^{(n)}(t+1)) -L(\boldsymbol{Q}^{(n)}(t))]$. Clearly,
  \begin{align}\label{eq:5}
 \Delta L  &= E\|\sqrt{\boldsymbol{\gamma}} \boldsymbol{Q}^{(n)}(t+1)\|  - \| \sqrt{\boldsymbol{\gamma}} \boldsymbol{Q}^{(n)}(t)\| \notag\\
  &\leq \frac{1}{2\| \sqrt{\boldsymbol{\gamma}} \boldsymbol{Q}^{(n)}(t)\| } \left( E\|\sqrt{\boldsymbol{\gamma}} \boldsymbol{Q}^{(n)}(t+1)\|^{2} -\| \sqrt{\boldsymbol{\gamma}} \boldsymbol{Q}^{(n)}(t)\|^{2} \right), 
  \end{align}
  where the inequality follows from the concavity of the function $\sqrt{x}$.
  Substitute the value of $\boldsymbol{Q}^{(n)}(t+1)$ from equation~(\ref{eq:1}), concentrate on the numerator of the above expression to obtain,
  \begin{align}\label{eq:6}
 &  E\|\sqrt{\boldsymbol{\gamma}} \boldsymbol{Q}^{(n)}(t+1)\|^{2}-\|\sqrt{\boldsymbol{\gamma}} \boldsymbol{Q}^{(n)}(t)\|^{2} \notag\\ 
  &= E\| \sqrt{\boldsymbol{\gamma}} \boldsymbol{Q}^{(n)}(t)+ \sqrt{\boldsymbol{\gamma}}\left( \boldsymbol{A}^{(n)}(t)-\boldsymbol{S}^{(n)}(t)+\boldsymbol{U}^{(n)}(t)\right)  \|^{2}\notag\\
&\qquad\qquad  -\|\sqrt{\boldsymbol{\gamma}} \boldsymbol{Q}^{(n)}(t)\|^{2}\notag\\
 &= E\left[ \|\sqrt{\boldsymbol{\gamma}}( \boldsymbol{A}^{(n)}(t)-\boldsymbol{S}^{(n)}(t)+\boldsymbol{U}^{(n)}(t)) \|^{2}\right.\notag\\
& \qquad \left.+ 2\left( \sqrt{\boldsymbol{\gamma}}\boldsymbol{Q}^{(n)}(t) \right)\cdot \left( \sqrt{\boldsymbol{\gamma}}\left( \boldsymbol{A}^{(n)}(t)-\boldsymbol{S}^{(n)}(t)+\boldsymbol{U}^{(n)}(t)\right)  \right)\right]\notag\\
  & = E\left[\|\sqrt{\boldsymbol{\gamma}}( \boldsymbol{A}^{(n)}(t)-\boldsymbol{S}^{(n)}(t)+\boldsymbol{U}^{(n)}(t)) \|^{2} \right.\notag\\
  &\qquad \left.+ 2\left(\boldsymbol{\gamma} \boldsymbol{Q}^{(n)}(t)\right)\cdot (\boldsymbol{A}^{(n)}(t)-\boldsymbol{S}^{(n)}(t)+\boldsymbol{U}^{(n)}(t)) \right]\notag\\
&= E\left[\|\sqrt{\boldsymbol{\gamma}}( \boldsymbol{A}^{(n)}(t)-\boldsymbol{S}^{(n)}(t)+\boldsymbol{U}^{(n)}(t))\|^{2} \right.\notag\\
&\left. + 2\left(\boldsymbol{\gamma} \boldsymbol{Q}^{(n)}(t)\right)\cdot \boldsymbol{U}^{(n)}(t) \right. \notag\\
& \qquad \qquad \qquad \left.+ 2\left(\boldsymbol{\gamma} \boldsymbol{Q}^{(n)}(t)\right)\cdot \left(\boldsymbol{A}^{(n)}(t)-\boldsymbol{S}^{(n)}(t)\right)\right]\notag\\
&\leq b_1 + b_2 + 2E\left[\left(\boldsymbol{\gamma} \boldsymbol{Q}^{(n)}(t)\right)\cdot \left(\boldsymbol{A}^{(n)}(t)-\boldsymbol{S}^{(n)}(t)\right) |\boldsymbol{Q}^{(n)}(t)\right],
\end{align}
where  $b_{1}$ is a uniform bound on $\|\sqrt{\boldsymbol{\gamma}}( \boldsymbol{A}^{(n)}(t)-\boldsymbol{S}^{(n)}(t)+\boldsymbol{U}^{(n)}(t)) \|^{2} $,
and 
  $b_2$ is a uniform bound on $\|2\left(\boldsymbol{\gamma} \boldsymbol{Q}^{(n)}(t)\right)\cdot \boldsymbol{U}^{(n)}(t)\|$ which
follows from the property that $U_i (t)>0$ only when $Q_i(t)$ is sufficiently small. 
  
To simplify exposition and avoid introducing additional notation, let us assume that
$\boldsymbol{\lambda}^{(n)} -\boldsymbol{\lambda}^{\star} = -\epsilon  \boldsymbol{\nu}$ for some $\epsilon>0$. 
(If not, then instead of $\boldsymbol{\lambda}^{\star}$ in this proof we can use $\boldsymbol{\lambda}^{\star\star}$, which the 
orthogonal projection of $\boldsymbol{\lambda}^{(n)}$ on $\bV^{\star}$.)
Combining~(\ref{eq:5}) and (\ref{eq:6}), we obtain
  \begin{align}\label{eq:7}
  2\|\sqrt{\boldsymbol{\gamma}} \boldsymbol{Q}^{(n)}(t)\|  \Delta L & \leq  b_{1} +b_{2}\notag \\
  &\quad + 2E\left[ \left(\boldsymbol{\gamma}\boldsymbol{Q}^{(n)}(t)\right)\cdot \left(\boldsymbol{A}^{(n)}(t)-\boldsymbol{S}^{(n)}(t)\right)\right]\notag\\
  & = b_{1} +b_{2}\notag \\
  &+ 2 E\left[\left(\boldsymbol{\gamma} \boldsymbol{Q}^{(n)}(t)\right)\right. \notag\\
   &\qquad \left. \cdot  \left(\boldsymbol{A}^{(n)}(t)-\boldsymbol{\lambda}^{\star} +\boldsymbol{\lambda}^{\star}-\boldsymbol{S}^{(n)}(t)\right)\right]\notag\\
  & = b_{1}+b_{2}- 2\epsilon  \|\left(\boldsymbol{\gamma} \boldsymbol{Q}^{(n)}(t)\right)_{\star}\|\notag\\
  & \quad+ 2 E\left[ \left(\boldsymbol{\gamma} \boldsymbol{Q}^{(n)}(t)\right)\cdot  \left(\boldsymbol{\lambda}^{\star}-\boldsymbol{S}^{(n)}(t)\right) \right]\notag\\
  & \leq b_{1}+b_{2} - 2\epsilon  \|\left(\boldsymbol{\gamma} \boldsymbol{Q}^{(n)}(t)\right)_{\star}\|\notag\\
  &\quad -\delta \|\left(\boldsymbol{\gamma} \boldsymbol{Q}^{(n)}(t)\right)_{\perp}\|,
  \end{align}
  where the last inequality follows from the definition of Max Weight (see ~(\ref{mwdef})) and the set $B^{\delta}_{\boldsymbol{\lambda}^{\star}}$ (see (\ref{deltaball})).
  If $\|\boldsymbol{\gamma} \boldsymbol{Q}^{(n)}(t)\| \ge x $, then at least one of $\|\left(\boldsymbol{\gamma} \boldsymbol{Q}^{(n)}(t)\right)_{\star}\|$ or $\|\left(\boldsymbol{\gamma} \boldsymbol{Q}^{(n)}(t)\right)_{\perp}\|$ is greater than or equal to $x/\sqrt{2}$. After some algebraic manipulations we obtain ($\gamma_{min}=\min_i \gamma_i$), 
  \begin{align*}
  \|\sqrt{\boldsymbol{\gamma}}\boldsymbol{Q}^{(n)}(t)\| > x  \implies \|\boldsymbol{\gamma} \boldsymbol{Q}^{(n)}(t)\| &>\sqrt{\gamma_{\min}} x \\
  \implies \|\left(\boldsymbol{\gamma} \boldsymbol{Q}^{(n)}(t)\right)_{\star}\| \vee \|\left(\boldsymbol{\gamma} \boldsymbol{Q}^{(n)}(t)\right)_{\perp} \| &\geq \frac{\sqrt{\gamma_{\min}} x}{\sqrt{2}}\\
  \implies \delta \|\left(\boldsymbol{\gamma} \boldsymbol{Q}^{(n)}(t)\right)_{\star}\| + \epsilon \|\left(\boldsymbol{\gamma} \boldsymbol{Q}^{(n)}(t)\right)_{\perp} \|& \geq \left(\epsilon \wedge \delta \right) \frac{\sqrt{\gamma_{\min}} x}{\sqrt{2}}.
  \end{align*}
  Substituting the above in inequality~(\ref{eq:7}) we see that the drift is upper bounded by 
$$
-\left(\epsilon \wedge \delta \right) \frac{\sqrt{\gamma_{\min}} x}{2\sqrt{2}} + \frac{b_{1}+ b_{2}}{\|\sqrt{\boldsymbol{\gamma}}\boldsymbol{Q}^{(n)}(t)\|}.
$$ 
This quantity is uniformly bounded by a negative constant for sufficiently large $x$. Application of Proposition~\ref{th:hajek} completes the proof.
 \end{proof}

  Now the positive recurrence of $\boldsymbol{Q}^{(n)}(\cdot)$ follows from 
Proposition~\ref{MTtheorem}. 
In fact, we will prove the following stronger statement. 
\begin{theorem}\label{theorem1}
For each $n=1,2,\ldots$, the Markov process $\boldsymbol{Q}^{(n)}(\cdot)$ is positive Harris recurrent and hence has a unique invariant probability distribution, which will be denoted $\chi^{(n)}$. 
Moreover, if $\boldsymbol{Q}^{(n)}(\infty)$ is the (random) process state in stationary 
regime (i.e. it has distribution $\chi^{(n)}$), 
\begin{align*}
E[ \| \bQ^{(n)}  (\infty) \|^{r}  ] <\infty , \forall r>0.
\end{align*}
\end{theorem}
\begin{proof}
By Lemma~\ref{prop2} any compact set is petite. Since $\boldsymbol{Q}^{(n)}(\cdot)$  is also bounded in probability (Lemma~\ref{prop1}), by Proposition~\ref{MTtheorem} $\boldsymbol{Q}^{(n)}(\cdot)$ is positive Harris recurrent. 

For a function $f(\cdot)$ and fixed $b>0$, denote $T_{b}f(\cdot)= f(\cdot) \wedge b$. 
Consider the process  starting from an arbitrary fixed initial state $\boldsymbol{Q}^{(n)}(0)$. Since the process is positive Harris recurrent, 
we can apply the ergodic theorem to obtain (note that $T_b \|\cdot\|$ is a bounded continuous function):
\begin{align}\label{mkeg1}
E\left(	T_{b}\|\boldsymbol{Q}^{(n)}(\infty)\|^{r}\right) = \lim_{m\to\infty}\frac{1}{m} \sum_{t=0}^{m}  E\left[T_{b} \|\boldsymbol{Q}^{(n)}(t) \|^{r}\right].
\end{align}
On the other hand,
\begin{align}\label{mkeg2}
\lim_{m\to\infty}\frac{1}{m} \sum_{t=0}^{m}  E\left[T_{b}\| \boldsymbol{Q}^{(n)}(t) \|^{r}\right] &\leq \lim_{m\to\infty}\frac{1}{m} \sum_{t=0}^{m} E\left[\| \boldsymbol{Q}^{(n)}(t)\|^{r}\right] \notag\\
&<C,
\end{align}
for some constant $C>0$,
where the second inequality follows from ~(\ref{hajek2}). 
Combining (\ref{mkeg1}) and~(\ref{mkeg2}), we have
\begin{align}\label{mkeg3}
E\left(	T_{b} \| \boldsymbol{Q}^{(n)}(\infty)  \|^{r}				\right) \leq C, \quad  \forall b>0,
\end{align}
and therefore, by monotone convergence theorem,
\begin{align*}
E\left( \| \boldsymbol{Q}^{(n)}(\infty)  \|^{r}				\right) = \lim_{b\to\infty} E\left(	T_{b} \| \boldsymbol{Q}^{(n)}(\infty)  \|^{r}				\right) \le C.
\end{align*}
\end{proof}
%\vspace{1cm}

\begin{lemma}
\label{lem-chi-abs-cont}
Uniformly on all (large) $n$ and the distributions of $\bQ^{(n)}(0)$, the distribution of $\bQ^{(n)}(1)$ is absolutely continuous w.r.t. Lebesgue measure, with the 
density upper bounded by $\delta^*$. 
%In particular, the latter properties hold for the distribution $\chi^{(n)}$  of $\bQ^{(n)}(\infty)$.
\end{lemma}

We omit the proof, which is straightforward, given our assumptions on the distribution of $\bA^{(n)}(t)$.

\begin{lemma}
\label{proposition3}
As $n\to\infty$, $\|\bQ^{(n)}(\infty)\| \to \infty$ in probability.
\iffalse
For any compact set $D\subset \mathbb{R}^{N}$,
 \begin{align}
 \lim_{n\to\infty}\chi^{(n)}(D) = 0.
 \end{align}
\fi
 \end{lemma}

 \begin{proof}
The proof is by contradiction. Suppose, for some fixed $C>0$ the compact set  
 $D = \{\boldsymbol{x}\in\mathbb{R}^{N} :\|\boldsymbol{x}\|\leq C \}$ is such that
 \begin{align}\label{proposition3:1}
\limsup\limits_{n\to\infty}\chi^{(n)}( D ) =\beta>0.
\end{align} 

Suppose $\bQ^{(n)}(t)\in D$. Then, using the same argument as in the proof of Lemma~\ref{prop2}, it is easy to see that for any $\epsilon>0$ there exists time $\tau \ge 1$,
such that $P\{\|\bQ^{(n)}(t+\tau)\| \le \epsilon\} \ge \beta_1>0$. This in turn implies that, with probability at least some $\beta_2>0$, for at least one flow $i$ the amount 
of wasted service $U_i^{(n)}(t+\tau) \ge \epsilon_2>0$. This implies that, for at least one $i$,
$$
\limsup\limits_{n\to\infty} E[U^{(n)}_{i}(\infty)] \geq \beta_1 \beta_2 \epsilon_2 >0.
$$
This, however, contradicts the fact that the process is stable for all large $n$.
\end{proof}

 \section{Steady-state queue lengths\\ deviations from $\bnu$}
\label{processYn}

Let us consider the process $\boldsymbol{Y}^{(n)}(\cdot)$,
defined as
$$
\boldsymbol{Y}^{(n)}(t) := (\bgamma \boldsymbol{Q}^{(n)}(t))_{\perp}.
$$
%i.e. the orthogonal projection of $\bgamma \boldsymbol{Q}^{(n)}(t)$ on $\bnu_{\perp}$. 

\begin{lemma}
\label{lemma3}
%The sequence $\Gamma^{(n)}$ is tight. Moreover, 
The steady-state expected norm $E\|\boldsymbol{Y}^{(n)}(\infty)\|$ is uniformly bounded in $n$.
\end{lemma}
\begin{proof}
As we did in the proof of Lemma~\ref{prop1}, to simplify exposition, assume that
$\boldsymbol{\lambda}^{(n)} -\boldsymbol{\lambda}^{\star} = -\epsilon  {\boldsymbol{\nu}}$. 
(If not, in this proof we would consider the projection $\blambda^{\star\star}$ of $\boldsymbol{\lambda}^{(n)}$ on $\bV^\star$, instead of $\blambda^\star$.
Consider Lyapunov function  $L(\boldsymbol{Q}) = \sum_{i=1}^{N}\gamma_{i}Q_{i}^{2}$.
By Theorem~\ref{theorem1}, $E L(\boldsymbol{Q}^{(n)}(\infty)) < \infty$.
The conditional drift of $L(\boldsymbol{Q})$ in one time step is given by (let $\boldsymbol{Q}^{(n)}(t)=\boldsymbol{Q}^{(n)}$,
$\boldsymbol{A}^{(n)}(t)=\boldsymbol{A}^{(n)}$, and so on,
to simplify notation)
 \begin{align}\label{driftineq}
 & E\left[L(\boldsymbol{Q}^{(n)}(t+1))  - L(\boldsymbol{Q}^{(n)}(t))|\boldsymbol{Q}^{(n)} \right] \notag \\
 & = E\left[\sum_{i=1}^{N}\gamma_{i}\left(Q^{(n)}_{i}+ A^{(n)}_{i}-S^{(n)}_{i} + U^{(n)}_{i}\right)^{2}|\boldsymbol{Q}^{(n)}\right] \notag\\
 &\qquad - \sum_{i=1}^{N}\gamma_{i}\left(Q^{(n)}_{i}\right)^{2}\notag\\
 &=E\left[\sum_{i=1}^{N}\gamma_{i}\left(A^{(n)}_{i}-S^{(n)}_{i} + U^{(n)}_{i}\right)\left(2Q^{(n)}_{i} \right.\right.\notag\\
&\quad \quad \quad \quad\quad \quad \quad \quad \quad \quad \quad \quad\left.\left. + A^{(n)}_{i}-S^{(n)}_{i} + U^{(n)}_{i}\right)  |\boldsymbol{Q}^{(n)}\right] \notag \\
 &=E\left[\sum_{i=1}^{N}\gamma_{i}\left(A^{(n)}_{i}-S^{(n)}_{i} + U^{(n)}_{i}\right)^{2}\right. \notag\\
&\qquad \qquad \qquad \left. + 2\gamma_{i} Q^{(n)}_{i} \left(A^{(n)}_{i}-S^{(n)}_{i} + U^{(n)}_{i}\right)|\boldsymbol{Q}^{(n)}\right] \notag \\
 &\leq b_{1} + 2 \left(\boldsymbol{\gamma} \boldsymbol{Q}^{(n)}\right)\cdot\left(\boldsymbol{\lambda}^{(n)} -E\left(\boldsymbol{S}^{(n)} |\boldsymbol{Q}^{(n)}\right) \right) \notag \\
 &= b_{1} + 2 \left(\boldsymbol{\gamma} \boldsymbol{Q}^{(n)}\right)\cdot\left(\boldsymbol{\lambda}^{(n)}-\boldsymbol{\lambda}^{\star}+\boldsymbol{\lambda}^{\star}-E\left(\boldsymbol{S}^{(n)} |\boldsymbol{Q}^{(n)} \right)\right) \notag \\
 &= b_{1} - 2\epsilon  \|\left(\boldsymbol{\gamma} \boldsymbol{Q}^{(n)}\right)_{\star}\|\notag\\
 &\qquad +2\left(\boldsymbol{\gamma} \boldsymbol{Q}^{(n)}\right)\cdot \left(\boldsymbol{\lambda}^{\star}-E\left(\boldsymbol{S}^{(n)}|\boldsymbol{Q}^{(n)}\right) \right) \notag \\
 &\leq b_{1} - 2\epsilon  \|\left(\boldsymbol{\gamma} \boldsymbol{Q}^{(n)}\right)_{\star}\|+2\min_{\boldsymbol{y}\in B^{\delta}_{\boldsymbol{\nu}} }\left(\boldsymbol{\gamma} \boldsymbol{Q}^{(n)}\right)\cdot \left(\boldsymbol{\lambda}^{\star}-\boldsymbol{y} \right) \notag \\
 &\leq b_{1} - 2\epsilon  \|\left(\boldsymbol{\gamma} \boldsymbol{Q}^{(n)}\right)_{\star}\|-2\delta \|\left(\boldsymbol{\gamma Q}\right)_{\perp}\|,
 \end{align}
where $b_{1}$ depends only on $\bgamma, \bA^{max}, \bS^{max}$, and the last inequality follows from the definition of MaxWeight and $B^{\delta}_{\boldsymbol{\lambda^{\star}}}$. 
Now consider the process $\boldsymbol{Q}^{(n)}(\cdot)$ in stationary regime, and take the expectation of both parts of~(\ref{driftineq}). We obtain, 
\begin{align}
 2\delta E\left[\|\left(\boldsymbol{\gamma Q}^{(n)}(\infty)\right)_{\perp}\|\right] + 2\epsilon  E\left[\|\left(\boldsymbol{\gamma Q}^{(n)}(\infty)\right)_{\star}\|\right] \leq b_{1}.
 \end{align}
Recalling that $\left(\boldsymbol{\gamma Q}^{(n)}(\infty)\right)_{\perp}=\boldsymbol{Y}^{(n)}(\infty)$,
we see that\\ $E\|\boldsymbol{Y}^{(n)}(\infty)\|$ is uniformly bounded.
\end{proof}

\section{Limit of the queue-differential process}
\label{sec-Ystar}

We now define a Markov chain $\boldsymbol{Y}^{\star}(\cdot)$, which, in the sense that will be made precise later, 
is a limit of the (non-Markov) process
$\boldsymbol{Y}^{(n)}(\cdot)$ as $n\to\infty$. 

Define $\boldsymbol{Y}^{(n)}(t)$ % $= (\bgamma \boldsymbol{Q}^{(n)}(t))_{\perp}$
as the orthogonal projection of $\bgamma \boldsymbol{Q}^{(n)}(t)$ on
the subspace $\bnu_{\perp}$. 
We call $\boldsymbol{Y}^{(n)}(\cdot)$ a {\em queue-differential}
process. (Obviously, {\em under the CRP condition}, the queue-differential process
is equal to the ``queue deviation'' process 
$\boldsymbol{Y}^{(n)}(\cdot)= (\bgamma \boldsymbol{Q}^{(n)}(t))_{\perp}$ in Section~\ref{processYn}.
When CRP does not hold, the ``deviation'' and ``differential'' processes
are defined differently. This will be discussed in Section~\ref{sec-no-crp}.) 
Denote by $\boldsymbol{Y}^{(n)}(\infty)$ the corresponding projection of 
the steady-state $\boldsymbol{Q}^{(n)}(\infty)$, and by $\Gamma^{(n)}$ its distribution.

Markov chain $\boldsymbol{Y}^{\star}(\cdot)$ is defined formally 
as follows. (We will show below that, in fact, the distribution $\Gamma^{(n)}$ converges to the stationary distribution
$\Gamma^{\star}$ of $\boldsymbol{Y}^{\star}(\cdot)$.)
   The state space of  $\boldsymbol{Y}^{\star}(\cdot)$ is  $\bnu_{\perp}$. 
Assume that at time $t$ the "scheduler" chooses decision 
\begin{align}
\label{modmw}
 k\in \argmax_{l:\boldsymbol{\mu}^{l}\in \boldsymbol{V}^{\star} } (\boldsymbol{Y^{\star}}(t)) \cdot \boldsymbol{\mu}^{l},
\end{align}
which determines the corresponding random amount of service $\bS(t)$, provided to the "queues" given by vector $\bQ^\star(t) = 
\bY^{\star}(t) / \bgamma$.
After that the (random) amount  $\bA^\star (t)$ of new "work" arrives and is added to the "queues."
Finally, the new queue lengths vector $\bQ^\star (t) - \bS(t) + \bA^\star (t)$ is transformed into $\bY^{\star}(t+1)$ via componentwise multiplication
by $\bgamma$ and orthogonal projection on $\bnu_{\perp}$.
(Note that both $\bQ^\star (t)$ and $ \bY^{\star}(t)$ may have components of any sign. Also, there is no "wasted service" here.)
In summary, the one step evolution is described by
\beql{eq777}
\boldsymbol{Y^{\star}}(t+1) = \boldsymbol{Y}^{\star}(t)
+ \left(\boldsymbol{\gamma} \boldsymbol{A}^{\star}(t) 
- \boldsymbol{\gamma}\boldsymbol{S}(t)\right)_{\perp}.
\end{equation}
Informally, one can interpret the process $\bY^{\star}(\cdot)$ as the 
queue-differential
process $\bY^{(n)}(\cdot)$, when $n$ is very large and the queue length vector
$\bQ^{(n)}$ is both large and has a small angle with $\bnu$. Under these
conditions, {\em the only service decisions $k$ that can be chosen are such that 
$\boldsymbol{\mu}^{k} \in \boldsymbol{V}^{\star}$, and the choice is uniquely determined 
by $\bY^{(n)}(\cdot)$.}
\iffalse
COM CHECK IF NEED THIS DEF We will assume that the service-decisions $k$ are numbered so that $\boldsymbol{\mu}^{k}\in \boldsymbol{V}^{\star}$ for  $k = 1,\ldots,J$ and $\boldsymbol{\mu}^{k}\notin \boldsymbol{V}^{\star}$ for $k = J+1,\ldots,K$. 
\fi

Let $\tilde{P}(\boldsymbol{x},\cdot)$ denote the one-step transition function for the Markov process $\boldsymbol{Y}^{\star}(\cdot)$.
If $\boldsymbol{x}\in \boldsymbol{\nu}_{\perp}$, then let $\tilde{B}_{\epsilon}(\boldsymbol{x})	:= \{ \boldsymbol{y}\in \boldsymbol{\nu}_{\perp}: ||\boldsymbol{y}-\boldsymbol{x}||\leq \epsilon\}$. The following fact is analogous to Lemma~\ref{lemma1}.

\begin{lemma}\label{yopdf}
(i) The points $\by \in \bnu_{\perp}$, such that 
\beql{eq-y-sched}
 k\in \argmax_{l:\boldsymbol{\mu}^{l}\in \boldsymbol{V}^{\star} } \by \cdot \boldsymbol{\mu}^{l}
\end{equation}
 is non-unique, form a set of zero Lebesgue measure.
Moreover, if $\by$ is such that  the corresponding decision $k$  is unique, then for a sufficiently small $\epsilon>0$ the decision $k$
is also the unique element of 
$$
 \argmax_{l:\boldsymbol{\mu}^{l}\in \boldsymbol{V}^{\star} } \bz \cdot \boldsymbol{\mu}^{l}
$$
for all $\bz \in \tilde B_{\epsilon}(\by)$.\\ 
(ii) There exist small $\epsilon>0$ and constant $c_*>0$, $c^*>0$ such that $\tilde{P}(\boldsymbol{x},\cdot)$ is absolutely
continuous and, moreover,
uniformly in $\bx \in \bnu_{\perp}$, the density of $\tilde{P}(\boldsymbol{x},\cdot)$ is
lower bounded by $c_*$ on set $\tilde{B}_{\epsilon}(\boldsymbol{x})$ and is upper bounded by $c^*$ everywhere.
\end{lemma}

\begin{proof}
%The proof is straightforward. 
Statement (i) is obvious. 
 Statement (ii) follows from our assumptions on the distribution of $\bA^{\star}(t)$, 
the fact that $\bA^{max} > \bS^{max}$, and the one-step
evolution rule \eqn{eq777}. We omit details.
\end{proof}

\begin{lemma}\label{yosmall}
For the Markov chain $\bY^{\star}(\cdot)$, every compact set is petite.
\end{lemma}

The proof easily follows from Lemma~\ref{yopdf}, by using the argument analogous to that in the proof of Lemma~\ref{prop2}.  We omit details.

%\section{Distribution $\Gamma^{(\star)}$}

Next, we establish some properties of a stationary distribution 
$\Gamma^{\star}$
of the  Markov process $\boldsymbol{Y}^{\star}(\cdot)$, {\em assuming a stationary distribution exists}.
This will help us later prove that the stationary distribution in fact exists and is unique.

\begin{lemma}\label{yocont}
If $\Gamma^{\star}$ is a  stationary distribution of $\boldsymbol{Y}^{\star}(\cdot)$, then  $\Gamma^{\star}$ is equivalent to the 
Lebesgue measure $\tilde{\mathcal{L}}$, i.e. 
$\Gamma^{\star} \ll \tilde{\mathcal{L}}$ and $\tilde{\mathcal{L}} \ll \Gamma^{\star}$.
\end{lemma}
\begin{proof}
$\Gamma^{\star} \ll \tilde{\mathcal{L}}$: This follows from Lemma~\ref{yopdf}.\\
$\tilde{\mathcal{L}} \ll \Gamma^{\star}$: It suffices to show that $\Gamma^{\star}(\tilde{B}_r (\boldsymbol{z}))>0$ for any $\bz \in \bnu_{\perp}$ and $r>0$. Consider the process $\boldsymbol{Y}^{\star}(\cdot)$ with the distribution of $\boldsymbol{Y}^{\star}(0)$ equal to $\Gamma^{\star}$.
(Then the process is of course stationary.) Fix any $0 < \beta < 1$ and choose a compact set ${D}\subset \bnu_{\perp}$ such that  $\Gamma^{\star}\left({D}\right)\geq \beta$. Using  Lemma~\ref{yopdf} we can easily show that there exists time $\tau>0$ and a constant $\Delta>0$,
such that, uniformly in $\bY^{\star}(0)=\bx \in D$,
$$
P\{ \bY^{\star}(\tau) \in \tilde{B}_r (\boldsymbol{z})   ~|~ \bY^{\star}(0)=\bx \} \ge \Delta,
$$
and therefore
$$
\Gamma^{\star}(\tilde{B}_r (\boldsymbol{z})) \geq \beta \Delta > 0.
$$
\end{proof}

\begin{lemma}\label{ne1}
Suppose $\Gamma^{\star}$ is a stationary distribution of $\boldsymbol{Y}^{\star}(\cdot)$. Then
$\tilde{P}_{\bx}(\boldsymbol{Y}^{\star} \to\infty  )  =0,~ \Gamma^{\star}- \mbox{a.s.}$, and hence $\tilde{P}_{\bx}(\boldsymbol{Y}^{\star} (t) \to\infty  )  =0,~  \tilde{\mathcal{L}}- \mbox{a.s.}$.
\end{lemma}
\begin{proof}
The proof is by contradiction.
Let $\boldsymbol{Y}^{\star}(0)$ have the stationary distribution $\Gamma^{\star}$, and assume that $\exists \quad \epsilon>0, \epsilon_1 >0$ such that 
\begin{align*}
\Gamma^{\star}(\{	\boldsymbol{x}: \tilde{P}_{\boldsymbol{x}}( \boldsymbol{Y}^{\star}\to\infty   )\geq \epsilon_1					\}) \geq \epsilon.
\end{align*}
This would imply that $\limsup\limits_{t\to\infty} {P}(\boldsymbol{Y}^{\star}(t)\in {D}     ) \leq 1-\epsilon \epsilon_1$ for every compact set ${D}\subset \bnu_{\perp}$. This is impossible, because the distribution of $\bY^{\star}(t)$ is equal to $\Gamma^{\star}$ for all $t$.
\end{proof}

\begin{lemma}\label{ne2}
If  process $\boldsymbol{Y}^{\star}(\cdot)$ has a stationary distribution,
it is non-evanescent.
\end{lemma}
\begin{proof}
Consider process $\boldsymbol{Y}^{\star}(\cdot)$ with fixed initial state\\  $\boldsymbol{Y}^{\star}(0) = \boldsymbol{x}$. Consider one-step transition. The distribution of $\boldsymbol{Y}^{\star}(1)$ is absolutely continuous with respect to $\tilde{\mathcal{L}}$. 
Thus, by Lemma~\ref{ne1}, with probability $1$, $\bz=\boldsymbol{Y}^{\star}(1)$ is such that $\tilde P_{\bz}(\bY^{\star} \to \infty)=0$.
Then, $\tilde P_{\bx}(\bY^{\star} \to \infty)=0$.
\end{proof}

 \begin{lemma}\label{rt12}
 Suppose $\Gamma^{\star}$ is a  stationary distribution of $\boldsymbol{Y}^{\star}(\cdot)$. Then, the Markov chain is positive Harris recurrent, and therefore
$\Gamma^{\star}$ is its unique stationary distribution.
 \end{lemma} 
\begin{proof}
Since every compact set is petite (Lemma~\ref{yosmall}) and the process is non-evanescent (Lemma~\ref{ne2}), 
it is Harris recurrent by Proposition~\ref{MTtheorem}. But since it has a finite invariant measure $\Gamma^{\star}$, $\boldsymbol{Y}^{\star}(\cdot)$ is positive Harris recurrent.
\end{proof} 

We now show the existence of a stationary distribution of  $\boldsymbol{Y}^{(\star)}(\cdot)$. 

\iffalse
\begin{lemma}
\label{lem-Yn-abs-cont}
Uniformly on all (large) $n$ and the distributions of $\bQ^{(n)}(0)$, 
the distribution of $\bY^{(n)}(1)$ is absolutely continuous w.r.t. Lebesgue measure, with the 
upper bounded density. 
%In particular, the latter properties hold for the distribution $\chi^{(n)}$  of $\bQ^{(n)}(\infty)$.
\end{lemma}

\begin{proof} This easily follows from Lemma~\ref{lem-chi-abs-cont} and the fact 
that $\|\bQ^{(n)}(1)-\bQ^{(n)}(0)\|$ is uniformly bounded.
\end{proof}
\fi

\begin{lemma}
\label{prop3}
 Every weak limit point $\Gamma^{(\star)}$ of the sequence of distributions $\Gamma^{(n)}$ is a stationary distribution of the process \\$\boldsymbol{Y}^{(\star)}(\cdot)$.
  \end{lemma}
\begin{proof}
Let $\Gamma^{\star}$ be a weak limit of $\Gamma^{(n)}$ along a subsequence on $n$. We can make the following observations.\\
(a) Observe that uniformly on all (large) $n$ and the distributions of $\bQ^{(n)}(0)$, 
the distribution of $\bY^{(n)}(1)$ is absolutely continuous w.r.t. Lebesgue measure, 
with the upper bounded density. (This easily follows from Lemma~\ref{lem-chi-abs-cont} and the fact 
that $\|\bQ^{(n)}(1)-\bQ^{(n)}(0)\|$ is uniformly bounded.) Then,
%Using Lemma~\ref{lem-Yn-abs-cont}
we see that $\Gamma^{\star}$ is absolutely continuous with bounded density. \\
(b) Consider any point $\by \in \bnu_{\perp}$ such that the decision $k$ in \eqn{eq-y-sched} is unique and a small $\epsilon>0$ such that
this decision $k$ is also unique for all $\bz \in \tilde B_{\epsilon}(\by)$. (See Lemma~\ref{yopdf}(i).) Then, there exists a sufficiently large $C>0$ such that, uniformly 
in $n$, conditions $\|\bQ^{(n)}(t)\| \ge C$ and $\bY^{(n)}(t) \in  \tilde B_{\epsilon}(\by)$ imply that the same decision $k$ will be unique at time $t$ for the process $\bQ^{(n)}(\cdot)$.

Using these two observations, Lemma~\ref{proposition3}, and
the fact that the distribution of $\bA^{(n)}(t)$ converges to that of $\bA^{\star}(t)$, 
we can choose a further subsequence of $n$ along which the following property holds.
The stationary versions of processes $\boldsymbol{Q}^{(n)}(\cdot)$ and the process
$\boldsymbol{Y}^{\star}(\cdot)$ with distribution of $\boldsymbol{Y}^{\star}(0)$ equal to $\Gamma^{\star}$,
can be constructed on one common probability space, so that with probability $1$:\\
(c) for all large $n$, the same decision $k$ is chosen at time $0$ 
in the processes $\bQ^{(n)}(\cdot)$ and $\bY^\star(\cdot)$;\\
(d) $\boldsymbol{Y}^{(n)}(0) \to  \boldsymbol{Y}^{\star}(0) \mbox{  and  }  \boldsymbol{Y}^{(n)}(1) \to  \boldsymbol{Y}^{\star}(1)$.
\iffalse
\begin{align*}
\boldsymbol{Y}^{(n)}(0) \to  \boldsymbol{Y}^{\star}(0) \mbox{  and  }  \boldsymbol{Y}^{(n)}(1) \to  \boldsymbol{Y}^{\star}(1).
\end{align*}
\fi

This, in turn, implies that for any bounded continuous function $g$ we have,
\begin{align*}
E\left[g(\boldsymbol{Y}^{\star}(0))\right] &= \lim_{n\to\infty} E\left[g(\boldsymbol{Y}^{(n)}(0))\right],\\
 E\left[g(\boldsymbol{Y}^{\star}(1))\right] &= \lim_{n\to\infty} E\left[g(\boldsymbol{Y}^{(n)}(1))\right]. 
\end{align*}
But, $E\left[g(\boldsymbol{Y}^{(n)}(0))\right]  = E\left[g(\boldsymbol{Y}^{(n)}(1))\right] $ for all $n$. Therefore, $E\left[g(\boldsymbol{Y}^{\star}(0))\right]  = E\left[g(\boldsymbol{Y}^{\star}(1))\right] $.
This proves stationarity of $\Gamma^{\star}$.
\end{proof}

 \begin{theorem}
\label{theorem2}
 The Markov process $\boldsymbol{Y}^{\star}(\cdot)$ is positive Harris recurrent. The sequence $\Gamma^{(n)}$ [i.e., the distributions of  $\boldsymbol{Y}^{(n)}(\infty)$]
weakly converges to the unique stationary distribution $\Gamma^{\star}$ of $\boldsymbol{Y}^{\star}(\cdot)$. 
 \end{theorem}
 \begin{proof}
This follows from Lemma~\ref{prop3} and Lemma~\ref{rt12}.
\end{proof}

We are finally in position to give a

\begin{proof}[of Theorem~\ref{th-smooth}]
By Theorem~\ref{theorem2}, the process $\boldsymbol{Y}^{\star}(\cdot)$ is positive Harris recurrent. Moreover,
we know that it is such that every compact set is petite. We can pick any compact set $D$ such that $\Gamma^\star(D)>0$,
and using Nummelin splitting view the process $\boldsymbol{Y}^{\star}(\cdot)$ as having an atom state, with finite average
return time to this atom. We see that the cumulative "service process" $\bG^\star(\cdot)$ corresponding to 
$\boldsymbol{Y}^{\star}(\cdot)$ in steady-state
is such that
\begin{align*}
\max_i \lim_{T\to\infty} P(G^{\star}_{i}(T)=0)=0.
\end{align*}
Finally, the argument used in the proof of Lemma~\ref{prop3} shows that the stationary versions of processes
$\boldsymbol{Y}^{\star}(\cdot)$ and $\boldsymbol{Q}^{(n)}(\cdot)$ for all (large) $n$ can be constructed 
on a common probability space in a way such that, w.p.1, for any $T>0$
$$
\bG^{(n)} (T) \to \bG^\star(T).
$$
This implies \eqn{eq888}.
\end{proof}

\section{Generalization to the case\\ when CRP condition does not\\ necessarily hold}
\label{sec-no-crp}

If CRP condition does not necessarily hold, let
$\bnu$ denote the {\em normal cone} to $\bV$ at point $\blambda^\star$; 
it has dimension $d\ge 1$.
(In the CRP case, $d=1$ and $\bnu$ is a ray.)
Fix any positive vector $\bnu'$ which lies in the relative interior of $\bnu$. Then, $\bV^\star$ is defined more generally as 
$$
\bV^\star = \argmax_{\bx\in \bV} \bnu' \cdot \bx;
$$
it is a $(N-d)$-dimensional face of $V$.
By $\bnu_\perp$ we denote the $(N-d)$-dimensional {\em subspace} orthogonal to $\bnu$.

We will denote by $\bx_\star$ the projection of a vector $\bx$ on the normal cone $\bnu$; that is, $\bx_\star$ is the closest
to $\bx$ point of $\bnu$. Then let $\bx_\perp = \bx - \bx_\star$, and let $\bx_{\perp,sp}$ be the orthogonal projection of
$\bx$ on the {\em subspace} $\bnu_\perp$. Note 
the difference between the definitions of $\bx_\perp$ and $\bx_{\perp,sp}$.
(In the CRP case, $\bx_\perp \equiv \bx_{\perp,sp}$. In the non-CRP case they are
in general different.) We always have 
$\|\bx_{\perp,sp}\| \le \|\bx_\perp\|$. Note that,
{\em if $\bx_\star$ lies in the relative interior of $\bnu$, then
$\bx_\star = \bx_{\perp,sp}$.}

In this notation, the entire development in Sections~\ref{sec-queue-length} 
and \ref{processYn} is carried out essentially as is,
with very minor adjustments.

The development in Section~\ref{sec-Ystar} is carried out with small adjustments,
which are as follows. The queue differential process is defined 
as $\boldsymbol{Y}^{(n)}(t) = (\bgamma \boldsymbol{Q}^{(n)}(t))_{\perp,sp}$.
Correspondingly, the one step evolution of $\bY^{\star}(\cdot)$ is defined by
\eqn{modmw} and
$$
\boldsymbol{Y^{\star}}(t+1) = \boldsymbol{Y}^{\star}(t)
+ \left(\boldsymbol{\gamma} \boldsymbol{A}^{\star}(t) 
- \boldsymbol{\gamma}\boldsymbol{S}(t)\right)_{\perp,sp}.
$$

Therefore, the state space for both $\boldsymbol{Y}^{(n)}(\cdot)$ and $\bY^{\star}(\cdot)$
is $\bnu_\perp$.

The proof of the key Lemma~\ref{prop3} requires, in addition to Lemma~\ref{proposition3},
the following Lemma~\ref{lem55555}.
Let $h(\bx)$ denote the distance from $\bx_\star$ to the relative boundary 
of the cone $\bnu$.
(To be precise, $h(\bx)$ is defined as the distance from $\bx_\star$ to the set \{relative boundary on the cone $\bnu$\}
$\setminus$ \{boundary of the positive orthant $\R^N_+$\}.)

\begin{lemma}
\label{lem55555}
As $n\to\infty$, $h(\bQ^{(n)}(\infty) \to \infty$ in probability.
\end{lemma}

This lemma is easily proved, because the contrary,
along with Lemmas~\ref{prop3} and Lemma~\ref{lemma3}, 
would imply that the frequency of choosing
scheduling decisions outside $\bV^\star$ would not vanish, as $n\to\infty$;
that would contradict stability when $n$ is large.

Then, in the proof of Lemma~\ref{prop3}, in the statement (b), the condition
$\|\bQ^{(n)}(t)\| \ge C$ is replaced by $h(\bQ^{(n)}(t)) \ge C$;
also, Lemma~\ref{lem55555} is used along with Lemma~\ref{proposition3}.

The statement of Theorem~\ref{theorem2} and the proof of  Theorem~\ref{th-smooth}
remain unchanged.

 \end{document}